\documentclass[10pt]{amsart}
\usepackage{amsmath,amssymb, amsthm, amsbsy}
\usepackage[dvips]{graphicx}
\usepackage[usenames,dvipsnames]{pstricks}
\usepackage{epsfig}
\usepackage{enumerate}
\usepackage{pst-grad} 
\usepackage{pst-plot} 
\usepackage{hyperref}

\newcommand{\R}{{\mathbb R}}
\newcommand{\N}{{\mathbb N}}

\newcommand{\C}{{\mathbb C}}

\newcommand{\A}{{\mathbf A}}

\renewcommand{\geq }{\geqslant}

\renewcommand{\leq }{\leqslant}

\numberwithin{equation}{section}
\newtheorem{Theorem}{Theorem}[section]

\newtheorem{Lemma}[Theorem]{Lemma}
\newtheorem{Proposition}[Theorem]{Proposition}
\theoremstyle{definition} 

\newtheorem{remark}[Theorem]{Remark}

\begin{document}
\title[Frequency-dependent time decay of Schr\"odinger flows]{Frequency-dependent time decay of\\ Schr\"odinger flows}
\author{Luca Fanelli}
\address{Luca Fanelli: SAPIENZA Universit$\grave{\text{a}}$ di Roma, Dipartimento di Matematica, P.le Aldo Moro 5, 00185, Roma, Italy}
\email{fanelli@mat.uniroma1.it}
\author{Veronica Felli}
\address{Veronica Felli: Universit$\grave{\text{a}}$ di Milano Bicocca,
Dipartimento di Scienza dei Materiali, Via Cozzi 55, 20125, Milano, Italy%
}
\email{veronica.felli@unimib.it}
\author{Marco A. Fontelos}
\address{Marco Antonio Fontelos: ICMAT-CSIC, Ciudad Universitaria de
Cantoblanco. 28049, Madrid, Spain}
\email{marco.fontelos@icmat.es}
\author{Ana Primo}

\address{Ana Primo: UAM, Ciudad Universitaria de Cantoblanco. 28049,
Madrid, Spain}
\email{ana.primo@uam.es}
\subjclass[2000]{35J10, 35L05.}
\keywords{Schr\"odinger equation, electromagnetic potentials, decay estimates}

\thanks{
The first author was supported by the project FIRB 2012:
``Dispersive dynamics: Fourier Analysis and Variational Methods'' funded by MIUR.
}

\subjclass[2000]{35J10, 35L05.}
\keywords{Schr\"odinger equation, electromagnetic potentials, representation
formulas, decay estimates}

\begin{abstract}
  We show that the presence of negative eigenvalues in
  the spectrum of the angular component of an electromagnetic
  Schr\"odinger hamiltonian $H$ generically produces a lack of the
  classical time-decay for the associated Schr\"odinger flow
  $e^{-itH}$. This is in contrast with the fact that dispersive
  estimates (Strichartz) still hold, in general, also in this case. We
  also observe an improvement of the decay for higher positive modes,
  showing that the time decay of the solution is due to the first
  nonzero term in the expansion of the initial datum as a series of
  eigenfunctions of a quantum harmonic oscillator with a singular
  potential.  A completely analogous phenomenon is shown for the heat
  semigroup, as expected.
\end{abstract}

\date{\today}
\maketitle

\section{Introduction}\label{sec:intro}
In this manuscript, we follow a research started in \cite{FFFP, f3p-2}
concerning on time decay of $L^p$-norms of solutions to scaling
invariant electromagnetic Schr\"odinger equations. In dimension
$N\geq2$, let us consider the hamiltonian
$$
H=\left(-i\nabla+ \dfrac{{\mathbf{A}}\big(\frac{x}{|x|}\big)} {|x|}
\right)^{\!\!2} + \dfrac{a\big(\frac{x}{|x|}\big)}{|x|^2},
$$
where $\mathbf{A}\in C^1(\mathbb S^{N-1};\R^N)$ is transversal, i.e.
\begin{equation}  \label{transversality}
{\mathbf{A}}(\theta)\cdot\theta=0 \quad \text{for all }\theta\in {\mathbb{S}}%
^{N-1},
\end{equation}
and $a\in
L^{\infty}({\mathbb{S}}^{N-1}; {\mathbb{R}})$. Here and in the sequel, we always denote by $r:=|x|$, $\theta=x/|x|$, so that $x=r\theta$. Associated to $H$, we study the Cauchy-problem for the Schr\"odinger equation
\begin{equation}\label{eq:schro}
\begin{cases}
  \partial_t u=  -iHu
  \\
  u(x,0)=u_0(x)\in L^2(\R^N),
  \end{cases}
\end{equation}
with $u=u(x,t):\R^{N+1}\to\C$.

A fundamental role in the description of the dynamics in \eqref{eq:schro} is played by the angular hamiltonian
\begin{equation}  \label{eq:angular}
L  =\big(-i\,\nabla_{\mathbb{S}^{N-1}}+{\mathbf{A}}\big)%
^2+a(\theta).
\end{equation}
Notice that $L$ is a symmetric operator, with compact inverse. Therefore, no continuous and residual spectrum are present, and
$$
\sigma(L) = \sigma_{\textrm{p}}(L)=\{\mu_1\leq\mu_2\leq\dots\}\subset\R,
$$
where the sequence $\{\mu_k\}$ diverges and each eigenvalue has finite multiplicity (see \cite[Lemma A.5]{FFT}).
For $k\in{\mathbb{N}}$, $k\geq 1$, we denote by $\psi_k$ the $L^{2}\big({\mathbb{S}}^{N-1},{\mathbb{C}}\big)$-normalized
eigenfunction of $L$ corresponding to $\mu_k$, namely
\begin{equation}  \label{angular}
\begin{cases}
L\psi_{k}=\mu_k\,\psi_k(\theta), & \text{in
}{\mathbb{S}}^{N-1}, \\[3pt]
\int_{{\mathbb{S}}^{N-1}}|\psi_k(\theta)|^2\,dS(\theta)=1. &
\end{cases}%
\end{equation}
By repeating each eigenvalue
as many times as its multiplicity, we can arrange the above enumeration in such a way that the correspondence $k\leftrightarrow\psi_k$ is one-to-one. Hence, normalizing, we can construct the set $\{\psi_k\}$ as an orthonormal basis in $L^2(\mathbb S^{N-1};\C)$.

The condition
\begin{equation}  \label{eq:hardycondition}
\mu_1>-\left(\frac{N-2}{2}\right)^{\!\!2}
\end{equation}
implies that
the quadratic form
$$
q[\psi]:=\int_{\R^N}\left|-i\nabla\psi+\frac{\mathbf{A}\left({x}/{|x|}\right)}{|x|}\psi\right|^2
+\int_{\R^N}\frac{a\left({x}/{|x|}\right)}{|x|^2}|\psi|^2,
$$
associated to $H$, is positive (in dimension $N=2$ by definition, while in dimension $N\geq3$ by magnetic Hardy inequality, (see \cite{lw}).
Therefore the hamiltonian $H$ is
realized as the self-adjoint extension (Friedrichs) of $q$ on the
natural form domain, and, by the Spectral Theorem, the hamiltonian flow $e^{-itH}$ associated to equation \eqref{eq:schro} is well defined.

Many efforts have been spent in the last decades to understand the dispersive properties of $e^{-itH}$.
In \cite{FFFP}, Theorem 1.3, we stated a useful representation formula which reads as follows
\begin{equation}\label{representation}
u(x,t)=\frac{e^{\frac{i|x|^{2}}{4t}}}{i(2t)^{{N}/{2}}}\int_{{\mathbb{R}}%
^{N}}K\bigg(\frac{x}{\sqrt{2t}},\frac{y}{\sqrt{2t}}\bigg)e^{i\frac{|y|^{2}}{%
4t}}u_{0}(y)\,dy,
\end{equation}
provided  \eqref{eq:hardycondition} holds. Here we denote
\begin{equation} \label{nucleo}
K(x,y)=\sum\limits_{k=1}^{\infty }i^{-\beta _{k}}j_{-\alpha
_{k}}(|x||y|)\psi _{k}\big(\tfrac{x}{|x|}\big)\overline{\psi _{k}\big(\tfrac{%
y}{|y|}\big)},
\end{equation}
where
\begin{equation} \label{eq:alfabeta}
  \alpha_k:=\frac{N-2}{2}-\sqrt{\bigg(\frac{N-2}{2}\bigg)^{\!\!2}+\mu_k},
  \quad \beta_k:=\sqrt{\left(\frac{N-2}{2}\right)^{\!\!2}+
    \mu_k},
\end{equation}
and, for every $\nu \in {\mathbb{R}}$,
\begin{equation*}
j_{\nu }(r):=r^{-\frac{N-2}{2}}J_{\nu +\frac{N-2}{2}}(r)
\end{equation*}%
with $J_{\nu }$ denoting the Bessel function of the first kind
\begin{equation*}
  J_{\nu }(t)=\bigg(\frac{t}{2}\bigg)^{\!\!\nu }\sum\limits_{k=0}^{\infty }
  \dfrac{(-1)^{k}}{\Gamma (k+1)\Gamma (k+\nu +1)}\bigg(\frac{t}{2}\bigg)
  ^{\!\!2k}.
\end{equation*}
As an immediate consequence of \eqref{representation}, we have the following:
\begin{equation}\label{eq:cor}
\sup_{x,y\in\R^N}\left|K(x,y)\right|<\infty
\qquad
\Rightarrow
\qquad
\left\|e^{-itH}\right\|_{L^1\to L^\infty}\leq C|t|^{-\frac N2},
\end{equation}
for some $C>0$ only depending on $N$,
where $\|e^{-itH}\|_{L^1\to L^\infty}$ denotes the norm of $e^{-itH}$
as an operator from $L^1$ into $L^\infty$. Although proving the uniform boundedness of $K$ can be hard, in \cite{FFFP, f3p-2} we can do it in the following cases:
\begin{itemize}
\item
if $N=2$, for generic $\mathbf{A}, a$ in the above class;
\item
if $N=3$, $\mathbf A\equiv 0$, and $0\leq a$-constant.
\end{itemize}
Starting from the dimension $N=3$, the negative range
\begin{equation}\label{eq:nega}
-\left(\frac{N-2}2\right)^2<\mu_1<0
\end{equation}
makes sense in this setting. Observe that, in the case $a$ constant
and $\A\equiv 0$, we have $\mu_1=a$. In the case $N=3$ and $\A\equiv0$
we have already proved that if $a\geq0$ then the classical $L^1$-$L^\infty$
time decay in \eqref{eq:cor} holds; hence it is natural to
wonder whether the $L^1$-$L^\infty$ time decay estimate still holds or not
 under condition \eqref{eq:nega}. We stress that Strichartz
estimates, which standardly follow by the $L^1$-$L^\infty$ bound, are known to
hold in this case, as proved in \cite{BPSTZ1, BPSTZ,
  PSTZ}. Nevertheless, the best which is known about time-decay for
the class of operators under consideration is in \cite{FFFP, f3p-2},
while for perturbative settings we refer to \cite{RS} as a standard
reference.

The first aim of this paper is to give a negative answer to the above
question, i.e., we want to show that condition \eqref{eq:nega}
immediately destroys the time-decay of the free flow. We can now state
our main result.
\begin{Theorem}\label{thm:schro}
  Let $N\geq3$, $a\in L^{\infty }({\mathbb{S}}^{N-1},{\mathbb{R}})$,
  ${\mathbf{A}}\in C^{1}({\mathbb{S}}^{N-1},{\mathbb{R}}^{N})$, and
  assume \eqref{transversality}, \eqref{eq:hardycondition}, and
  \eqref{eq:nega}. Then, for almost every $t\in\R$,
  $e^{-itH}(L^1)\not\subseteq L^\infty$; in particular
  $e^{-itH}$ is not a
  bounded operator from $L^1$ to $L^\infty$.
\end{Theorem}
\begin{remark}
  Condition \eqref{eq:nega} needs to be read, eventually, in terms of
  the usual Hardy inequality, and its extension to the case of
  magnetic derivatives (see the standard reference \cite{lw}).
  Theorem \ref{thm:schro} shows that condition $\mu_1\geq0$ is
  necessary for the $L^1\to L^\infty$ bound.  In other words, if the
  spherical Hamiltonian is not positive, the usual time decay property
  does not hold.  As we see in the sequel, in the case $\mu_1\geq0$
  the rate of time decay (in suitable topologies) depends on the size
  of $\mu_1$.
\end{remark}
The hint for the proof of Theorem \ref{thm:schro} comes from estimate
(1.29) in \cite{FFFP}, which suggests the failure of the
$L^1$-$L^\infty$ decay in the case of the inverse square
potential. The proof of Theorem \ref{thm:schro} shows that estimate
(1.29) is sharp, and that the phenomenon is general, and related to
the existence of the negative energy-level $\mu_1<0$.  The key role in
the proof is played by the operator
$$
T:=H+\frac14|x|^2,
$$
whose spectral properties are described in Section \ref{sec:schro}
below. Since $T$ has discrete spectrum, and we can decompose
$L^2(\R^N)$ as a direct sum of eigenspaces for $T$, we can expand the
initial datum $u_0$ for \eqref{eq:schro} as a series of the following eigenfunctions
of $T$  forming an orthogonal basis of $L^2(\R^N)$:
\begin{equation}  \label{eigenvectors}
V_{n,j}(x)= |x|^{-\alpha_j}e^{-\frac{|x|^2}{4}}P_{j,n}\Big(\frac{|x|^2}{2}%
\Big) \psi_j\Big(\frac{x}{|x|}\Big),\quad n,j\in\N,\ j\geq1,
\end{equation}
where $P_{j,n}$ is the polynomial of
degree $n$ given by
\begin{equation*}
P_{j,n}(t)=\sum_{i=0}^n \frac{(-n)_i}{\big(\frac{N}2-\alpha_j\big)_i}\,\frac{%
t^i}{i!},
\end{equation*}
denoting as $(s)_i$, for all $s\in{\mathbb{R}}$, the Pochhammer's symbol $%
(s)_i=\prod_{j=0}^{i-1}(s+j)$, $(s)_0=1$.
The main argument in the proof of Theorem \ref{thm:schro} is that the evolution of those
eigenfunctions, as initial data for \eqref{eq:schro}, is quite
explicit.

The second purpose of the present paper is to prove, when the classical
time decay holds, e.g. in the case  $a\geq 0$ constant
and $\A\equiv 0$, an improvement of the decay for higher positive modes.
 Roughly speaking, the more
  positive is $\mu_1>0$, the faster decay is expected to be, in
  suitable topologies (see e.g. \cite{CK, FGK, gk, ko1, ko2, kr,
    krz} for some recent works related to this topic, both for Schr\"odinger
  and heat flows).
For all $k> 1$, let us denote as
\[
\mathcal U_k=\mathop{\rm span}\left\{V_{n,j}:n\in\N,1\leq j<
  k\right\}\subset L^2(\R^N).
\]
In the following theorem we observe that the time decay of the
solution is due to the first nonzero term in the expansion of the
initial datum as a series of eigenfunctions \eqref{eigenvectors}.

\begin{Theorem}\label{thm:2}
Let $N=3$,  $a\geq 0$, and define $H=-\Delta+\frac{a}{|x|^2}$.
\begin{enumerate}[\rm (i)]
\item There exists $C>0$ such that,
for all $f\in L^2(\R^3)$ with $|x|^{-\alpha_1}f\in L^1(\R^3)$,
\begin{equation*}
\left\||x|^{\alpha_1}e^{-itH}f(\cdot)\right\|_{L^\infty}\leq C
t^{-\frac 32+\alpha_1}
\||x|^{-\alpha_1}f\|_{L^{1}}.
\end{equation*}
\item  For all
$k\in\N$, $k\geq1$, there exists $C_k>0$ such that,
for all $f\in \mathcal U_k^\perp$ with $|x|^{-\alpha_k}f\in L^1(\R^3)$,
\begin{equation*}
\left\||x|^{\alpha_k}e^{-itH}f(\cdot)\right\|_{L^\infty}\leq C_k
t^{-\frac 32+\alpha_k}
\||x|^{-\alpha_k}f\|_{L^{1}}.
\end{equation*}
\end{enumerate}
\end{Theorem}

The rest of the paper is devoted to the proof of Theorems
\ref{thm:schro} and \ref{thm:2}. A final section is devoted to the description of the same phenomenon for the electromagnetic heat flow $e^{-tH}$, which enjoys the same scaling invariance of $e^{-itH}$.

\section{Proof of Theorems \ref{thm:schro} and \ref{thm:2}}\label{sec:schro}
The proof of Theorems \ref{thm:schro} is constructive: assuming \eqref{eq:nega}, we can construct
an explicit initial datum $u_0\in L^1$ such that $e^{-itH}u_0\notin
L^\infty$. The argument strongly relies on the strategy which leads to
the representation formula \eqref{representation}. In order to do
this, we start with some preliminaries, concerning the functional
setting of our problem. The following is analogous to Section 2 in
\cite{FFFP}. We write it here for the sake of completeness (see also
\cite{FFT} for further details).

Define the following Hilbert spaces:
\begin{itemize}
\item the completion ${\mathcal{H}}$ of $C^{\infty}_{\mathrm{c}%
}({\mathbb{R}}^N\setminus\{0\},{\mathbb{C}})$ with respect to the norm
\begin{equation*}
\|\phi\|_{{\mathcal{H}}}=\bigg(\int_{{\mathbb{R}}^N}\bigg(|\nabla\phi(x)|^2+
\Big(|x|^2+\frac1{|x|^2}\Big)|\phi(x)|^2\bigg) \,dx\bigg)^{\!\!1/2};
\end{equation*}

\item the completion $\widetilde{\mathcal H}$ of $C^{\infty}_{\mathrm{c}}({\mathbb{%
R}}^N,{\mathbb{C}})$ with respect to the norm
\begin{equation*}
\|\phi\|_{\widetilde{\mathcal H}}=\bigg(\int_{{\mathbb{R}}^N}\Big(|\nabla\phi(x)|^2+  \big(|x|^2+1%
\big)|\phi(x)|^2\Big) \,dx\bigg)^{\!\!1/2};
\end{equation*}

\item the completion ${\mathcal{H}}_A$ of $%
C^{\infty}_{\mathrm{c}}({\mathbb{R}}^N\setminus\{0\},{\mathbb{C}})$ with
respect to the norm
\begin{equation*}
\|\phi\|_{{\mathcal{H}}_{\mathbf{A}}}=\bigg(\int_{{\mathbb{R}}^N}\Big(%
|\nabla_{\mathbf{A}}\phi(x)|^2+  \big(|x|^2+1\big)|\phi(x)|^2\Big) \,dx\bigg)%
^{\!\!1/2}
\end{equation*}
with $\nabla_{{\mathbf{A}}}\phi= \nabla\phi+i\,\frac {{\mathbf{A}}(x/|x|)}{%
|x|}\phi$.
\end{itemize}
It is clear that $\mathcal{H }\hookrightarrow \widetilde{\mathcal H}$, with continuous embedding.
This, together with \cite[Proposition 6.1]{KW}, gives in addition that
$\mathcal H\hookrightarrow L^p(\R^n)$, with compact embedding,
for all
\begin{equation*}
2\leq p<
\begin{cases}
2^{*}=\frac{2N}{N-2}, & \text{if }N\geq3, \\
+\infty, & \text{if }N=2.%
\end{cases}
\end{equation*}
In analogy with \cite{FFFP}, in the next result,
by a pseudoconformal change of variables (see e.g. \cite{KW}), we reduce the hamiltonian $H$ in equation \eqref{eq:schro} to a new operator, with an harmonic oscillator involved.
\begin{Lemma}
Let \eqref{eq:hardycondition} hold and $u=e^{-itH}u_0$. Then
\begin{equation}  \label{varphi}
\varphi(x,t)= (1+t^2)^{\frac{N}{4}}u\big(\sqrt{1+t^2}x,t\big)e^{-it\frac{%
|x|^2}{4}}
\end{equation}
satisfies
\begin{align*}
&\varphi\in C({\mathbb{R}}; L^{2}({\mathbb{R}}^N)),\quad \varphi (x, 0)=
u(x,0), \\
&\|\varphi(\cdot,t)\|_{L^{2}({\mathbb{R}}^N)}=\|u(\cdot,t)\|_{L^{2}({\mathbb{%
R}}^N)} \text{ for all }t\in{\mathbb{R}},
\end{align*}
and
\begin{equation}  \label{varphieq}
i\dfrac{d\varphi}{d t}(x,t)= \dfrac{1}{(1+t^2)} \bigg(H\varphi(x,t)+\frac{1}{4}|x|^2 \varphi(x,t)\bigg).
\end{equation}
\end{Lemma}
Denote now by
\begin{equation}  \label{operator}
T:{\mathcal{H}}\to {\mathcal{H}}^\star,\quad T=H+\frac{1}{4}|x|^2,
\end{equation}
naturally defined via the associated (positive) quadratic form.
Assumption \eqref{eq:hardycondition}
gives
\begin{equation}  \label{eq:posdef}
\int_{{\mathbb{R}}^N} \bigg[ \big| \nabla_{\mathbf{A}}\phi(x)\big|^2 -\frac{a%
\big({x}/{|x|}\big)}{|x|^2}|\phi(x)|^2+\frac{|x|^2}4|\phi(x)|^2 \bigg]\,dx
\geq C(N,{\mathbf{A}},a)\|\phi\|_{\mathcal{H}}^2,
\end{equation}
for some $C=C(N,\mathbf A,a)>0$ and for all $\phi\in\mathcal{H}$.

The following proposition, which describes completely the spectrum of
$T$, was proved in \cite[Proposition 3.2]{FFFP}.
\begin{Proposition}\label{Hilbert}
Let ${\mathbf{A}}\in C^1({\mathbb{S}}^{N-1},{\mathbb{R}}^N)$
and $a\in  L^{\infty}\big({\mathbb{S}}^{N-1}\big)$, and assume
\eqref{eq:hardycondition}.
Then
$$
\sigma(T)=\sigma_{\textrm{p}}(T)
=
\big\{ \gamma_{m,k}: k,m\in{\mathbb{N}}, k\geq 1\big\}
$$
where
\begin{equation}  \label{eigenvalues}
\gamma_{m,k}=2m-\alpha_k+\dfrac N2, \quad \alpha_k=\frac{N-2}{2}-\sqrt{\bigg(%
\frac{N-2}{2}\bigg)^{\!\!2}+\mu_k},
\end{equation}
$\mu_k$ is the $k$-th eigenvalue of the operator $L$ on $L^2(\mathbb{S}^{N-1})$ and $\alpha_j,\beta_j$ are defined in \eqref{eq:alfabeta}. Each eigenvalue $%
\gamma_{m,k}$ has finite multiplicity equal to
\begin{equation*}
\#\bigg\{j\in{\mathbb{N}},j\geq 1: \frac{\gamma_{m,k}}{2}+\frac{\alpha_j}%
2-\frac N4\in{\mathbb{N}}\bigg\}
\end{equation*}
and a basis of the corresponding eigenspace is
\begin{equation*}
\left\{V_{n,j}: j,n\in{\mathbb{N}},j\geq  1,\gamma_{m,k}=2n-\alpha_j+\frac
N2 \right\},
\end{equation*}
where $V_{n,j}$ is defined in \eqref{eigenvectors}.
\end{Proposition}

\begin{remark}
\label{rem:ortho}  It is easy to check that
\begin{equation*}
\text{if }(m_1,k_1)\neq(m_2,k_2)\quad\text{then}\quad V_{m_1,k_1}\text{ and }
V_{m_2,k_2}\text{ are orthogonal in }{\  L^{2}({\mathbb{R}}^N)}.
\end{equation*}
By Proposition \ref{Hilbert} and classical spectral theory, it follows that
\begin{equation*}
\left\{ \widetilde V_{n,j}= \frac{V_{n,j}}{\|V_{n,j}\|_{L^{2}({\mathbb{R}}%
^N)}}: j,n\in{\mathbb{N}},j\geq  1\right\}
\end{equation*}
is an orthonormal basis of $L^{2}({\mathbb{R}}^N)$.
\end{remark}
Notice, in addition, that $\widetilde{V}_{n,j}\in L^1(\R^N)$, since $\alpha_j<N$, for any $j$, by \eqref{eq:alfabeta}. In order to prove our main result, it is fundamental to study the solution of \eqref{eq:schro}, with initial datum $\widetilde V_{n,j}$.
\begin{Theorem}\label{Counterexample}
  Let $a\in L^{\infty }({\mathbb{S}}^{N-1},{\mathbb{R}})$,
  ${\mathbf{A}}\in C^{1}({\mathbb{S}}^{N-1},{\mathbb{R}}^{N})$, and
  assume \eqref{transversality} and  \eqref{eq:hardycondition}. Then, for
  any $n,j\in\mathbb N$, $j\geq1$,
  \begin{multline}\label{eq:jmode}
    e^{-itH}\widetilde V_{n,j}(x)\\=
    (1+t^2)^{-\frac{N}{4}+\frac{\alpha_j}{2}}\,|x|^{-\alpha_j}
    \frac{e^{\frac{-|x|^2}{4(1+t^2)}}}{\|V_{n,j}\|_{L^{2}(\mathbb{R}^{N})}}
    e^{i\frac{|x|^2 t}{4(1+t^2)}} e^{-i\gamma_{n,j}\arctan t}
    \psi_j\big(\tfrac{x}{|x|}\big)
    P_{n,j}\big(\tfrac{|x|^2}{2(1+t^2)}\big).
\end{multline}
\end{Theorem}
\begin{proof}
Let $u_0=\widetilde V_{n,j}$. We notice that, since $2\alpha_j<N$,
$u_0\in L^2$, hence we can expand it in Fourier series with respect to
the orthonormal system $\big\{ \widetilde V_{m,k}:
  m,k\in{\mathbb{N}},k\geq  1\big\}$, i.e.
\begin{equation}\label{datoinicial}
u_{0}=\sum\limits_{\substack{ m,k\in {\mathbb{N}} \\ k\geq 1}}c_{m,k}
\widetilde{V}_{m,k}
\quad\text{in }L^2
\end{equation}
with
\[
c_{m,k}=\int_{{\mathbb{R}}^{N}}u_{0}(x)\overline{\widetilde{V}_{m,k}(x)}
\,dx
=
\begin{cases}
1,
&
\text{if }(m,k)=(n,j),
\\
0,
&
\text{otherwise}.
\end{cases}
\]
In a similar way, for $t>0$, we can expand the function $\varphi (\cdot ,t)$ defined
in \eqref{varphi} (with $u=e^{-itH}u_0$) as
\begin{equation}
\varphi (\cdot ,t)=\sum\limits_{\substack{ m,k\in {\mathbb{N}} \\ k\geq 1}}%
\varphi _{m,k}(t)\widetilde{V}_{m,k}\quad \text{in }L^{2}({\mathbb{R}}^{N}),
\label{eq:exp_varphi}
\end{equation}%
where
\begin{equation*}
\varphi _{m,k}(t)=\int_{{\mathbb{R}}^{N}}\varphi (x,t)\overline{\widetilde{V}%
_{m,k}(x)}\,dx.
\end{equation*}%
Since $\varphi (z,t)$ satisfies \eqref{varphieq}, we obtain that $\varphi
_{m,k}\in C^{1}({\mathbb{R}},{\mathbb{C}})$ and
\begin{equation*}
i\varphi _{m,k}^{\prime }(t)=\dfrac{\gamma _{m,k}}{1+t^{2}}\varphi
_{m,k}(t),\quad \varphi _{m,k}(0)=c_{m,k},
\end{equation*}%
which by integration yields $\varphi _{m,k}(t)=c_{m,k}e^{-i\gamma
_{m,k}\arctan t}$. Hence expansion (\ref{eq:exp_varphi}) can be rewritten as
\begin{equation*}
\varphi (z,t)=\sum\limits_{\substack{ m,k\in {\mathbb{N}} \\ k\geq 1}}%
c_{m,k}e^{-i\gamma _{m,k}\arctan t}\widetilde{V}_{m,k}(z)
=
e^{-i\gamma_{n,j}\arctan t}\widetilde{V}_{n,j}(z),
\end{equation*}%
for all $t>0$.
It follows by \eqref{varphi} that
\begin{align*}
&e^{-itH}\widetilde V_{n,j}(x)
= e^{i\frac{t|x|^2}{4(1+t^2)}} (1+t^2)^{-\frac{N}{4}} \varphi \left(  \dfrac{x}{\sqrt{1+t^2}}, t\right)
\\
&
=    (1+t^2)^{-\frac{N}{4}+\frac{\alpha_j}{2}}\,|x|^{-\alpha_j}
    \frac{e^{\frac{-|x|^2}{4(1+t^2)}}}{\|V_{n,j}\|_{L^{2}(\mathbb{R}^{N})}}
    e^{i\frac{|x|^2 t}{4(1+t^2)}} e^{-i\gamma_{n,j}\arctan t}
    \psi_j\big(\tfrac{x}{|x|}\big)
    P_{n,j}\big(\tfrac{|x|^2}{2(1+t^2)}\big),
\end{align*}
and the proof is now complete.
\end{proof}
\begin{proof}[Proof of Theorem \ref{thm:schro}]
  The proof of Theorem \ref{thm:schro} is now a straightforward
  consequence of the previous result. It is sufficient to notice that
  condition \eqref{eq:nega} implies that $\alpha_1>0$, by
  \eqref{eq:alfabeta}. Therefore, taking as initial datum
  $u_0=\widetilde V_{0,1}$ for \eqref{eq:schro} and observing that
  $P_{1,0}\equiv 1$, formula
  \eqref{eq:jmode} gives that $e^{-itH}u_0\notin L^\infty$.
\end{proof}
\begin{remark}
  Formula \eqref{eq:jmode} also gives a quite precise description of the time-decay phenomenon of each Fourier mode. Denote by $\pi_{n,j}$ the projector of $L^2(\R^N)$ on the eigenspace generated by $V_{n,j}$, namely
  $$
  \pi_{n,j}f:=c_{n,j}\widetilde V_{n,j},
  \qquad
  c_{n,j}:=\int_{\R^N}f(x)\overline{\widetilde V_{n,j}(x)}\,dx,
  $$
  and assume that $\mu_j\geq0$, which implies by \eqref{eq:alfabeta} that $\alpha_j\leq0$.
  Identity \eqref{eq:jmode} hence easily gives
  \begin{equation}\label{eq:dia}
  \left\||x|^{-\gamma_j}e^{-itH}\pi_{n,j}u_0\right\|_{L^\infty}
  \leq
  C|t|^{-\frac N2-\gamma_j}\left\||x|^{\gamma_j}\pi_{n,j}u_0\right\|_{L^1},
  \qquad
  \gamma_j:=-\alpha_j,
  \end{equation}
  for some $C>0$ independent on $u_0$ (but depending on $j$ and $n$).
  In particular, in the weighted topologies $L^1(|x|^{\gamma_j}\,dx),\
  L^\infty(|x|^{-\gamma_j}\,dx)$, the decay of the evolution $j$-th
  mode of the initial datum, with respect to $V_{n,j}$, is faster than
  the usual one by a polynomial factor $\alpha_j$. This perfectly
  matches the diamagnetic phenomenon already observed in \cite{FGK,
    gk}. Notice that, in the free case $\mathbf A\equiv a\equiv0$, we
  have $\mu_1=0$ and the leading term for the time decay is given by
  the 0-th mode, which gives the usual decay rate $|t|^{-\frac
    N2}$. Nevertheless, proving a time-decay estimate for $e^{-itH}$
  is both a matter of zero modes, which can be rephrased by the
  condition $\mu_1\geq0$, and of high modes, which is concerned with
  the uniform convergence of a series with generic term as in
  \eqref{eq:dia}.
\end{remark}

 In the case $N=3$, $\mathbf A\equiv 0$ and $a\geq0$  in
which the series in \eqref{nucleo} is bounded and the classical time
decay consequently holds, the improved time decay observed in \eqref{eq:dia} for initial data of
type \eqref{eigenvectors} can be extended to initial data which are
orthogonal to the space generated by the first modes, as stated in
Theorem \ref{thm:2}.

\begin{proof}[Proof of Theorem \ref{thm:2}]
If the initial datum $u_0$ in \eqref{eq:schro} belongs to $\mathcal
U_k^\perp$ (in the case $k=1$ if  $u_0\in L^2(\R^3)$), it can be expanded in Fourier series  as
\begin{equation*}
u_0=\sum\limits_{\substack{ m,j\in {\mathbb{N}} \\ j\geq k}}c_{m,j}
\widetilde{V}_{m,k}\quad \text{in }L^{2}({\mathbb{R}}^{3}),\quad \text{where
}c_{m,j}=\int_{{\mathbb{R}}^{3}}u_0(x)\overline{\widetilde{V}_{m,j}(x)}
\,dx.
\end{equation*}
Then all the series expansions appearing in the proof of \cite[Theorem
1.3]{FFFP} start from the index $k$ instead of $1$. Therefore, if the
initial datum $u_0$ belongs to $\mathcal U_k^\perp$,
 the
representation formula \eqref{representation} can be refined as
\begin{equation}\label{representation_refined}
u(x,t)=\frac{e^{\frac{i|x|^{2}}{4t}}}{i(2t)^{{3}/{2}}}\int_{{\mathbb{R}}%
^{3}}K_k\bigg(\frac{x}{\sqrt{2t}},\frac{y}{\sqrt{2t}}\bigg)e^{i\frac{|y|^{2}}{%
4t}}u_{0}(y)\,dy,
\end{equation}
where
\begin{equation*} \label{nucleo_k}
K_k(x,y)=\sum\limits_{j\geq k}i^{-\beta _{k}}j_{-\alpha
_{k}}(|x||y|)\psi _{k}\big(\tfrac{x}{|x|}\big)\overline{\psi _{k}\big(\tfrac{%
y}{|y|}\big)}.
\end{equation*}
From \cite[Proof of Theorem 1.1]{FFFP} it follows that 
\begin{equation}\label{eq:1}
|K_k(x,y)|\leq C (|x||y|)^{-\alpha_k}\quad
\text{if $|x||y|\leq
\delta$,}
\end{equation}
for some constant $C>0$ depending on $\delta$ and $k$ but independent
of $x,y$; furthermore from \cite[estimate (6.16)]{FFFP} it follows
that
\begin{equation}\label{eq:2}
\sup_{x,y\in \R^3}|K_k(x,y)|<+\infty.
\end{equation}
Combining \eqref{eq:1} and \eqref{eq:2} and recalling that
$\alpha_k<0$, we conclude that
\begin{equation}\label{eq:3}
|K_k(x,y)|\leq {\rm const\,}(|x||y|)^{-\alpha_k}
\end{equation}
for all $x,y\in \R^2$, for some constant ${\rm const\,}>0$ independent
of $x,y$. The conclusion follows from estimate \eqref{eq:3} and the
representation formula \eqref{representation_refined}.
\end{proof}

\appendix

\section{Heat semigroup}\label{sec:heat}
We conclude this note with a small appendix devoted to the heat
semigroup $e^{-tH}$. Since it enjoys the same scaling invariance of
$e^{-itH}$, it is natural to expect analogous phenomena occurring
 at the level of $L^p-L^{p'}$ time-decay. Moreover, notice that
by the Barry Simon's diamagnetic inequality
$$
\left|e^{-tH}f\right|\leq \left|e^{t\Delta}f\right|,
\qquad
\text{provided }a(\theta)\geq0,
$$
for all $t>0$
(see \cite{S}). Hence we easily obtain that
$$
\left\|e^{-tH}u_0\right\|_{L^\infty}\leq C|t|^{-\frac 32}\left\|u_0\right\|_{L^1},
$$
for all $t>0$, provided $a(\theta)\geq0$, as a consequence of the same
estimate for $e^{-t\Delta}$.  In the general case (i.e. if $a$ has any
sign), the phenomenon of lack of the classical $L^1-L^\infty$ bound
was completely described in \cite{IIY}, where sharp decay estimates of
$L^q$-norms for nonnegative Schr\"odinger heat semigroups were
established.
In the spirit of Theorem \ref{Counterexample}, we construct below an explicit
example  confirming the lack of the classical $L^1-L^\infty$ bound
proved in \cite{IIY}.
Here
the situation is quite simpler, due to the following self-similarity
issue.  Let us consider the equation
\begin{equation}\label{eq:homogeneous}
v_t-\Delta v+a\dfrac{\,v}{|x|^{2}}=0,
\quad
a>-\bigg(\frac{N-2}{2}\bigg)^{\!\!2}.
\end{equation}
We look for solutions of the form
$$
v(x,t)= t^{-\mu}\phi\Big(\frac{r}{t^{\nu}}\Big)\psi_k(\theta), \quad
\text{where }x=r\theta,\ r\geq0,\ \theta\in {\mathbb S}^{N-1},
$$
and $\psi_k$ is as in \eqref{angular} with
$\A\equiv 0$ for some $k\geq1$.
By the change of variables $s=\frac{r}{t^{\nu}}$ and $\nu=\frac{1}{2}$, we obtain
\begin{multline}\label{eq:phi}
0=v_t-v''-\Big(\dfrac{N-1}{r}\Big)v'+\dfrac{a}{r^2}v-\frac1{r^2}\Delta_{{\mathbb
    S}^{N-1}}v\\\equiv
-t^{-\mu-1}\psi_k(\theta)\bigg(\phi''(s)+\Big(\dfrac{N-1}{s}+\dfrac{s}{2}\Big)\phi'(s)+\Big(\mu-\dfrac{\mu_k}{s^2}\Big)\phi(s)\bigg),
\end{multline}
where $\mu_k$ is the $k$-th eigenvalue of problem \eqref{angular}.
Denoting by $\phi(s)=s^{-\alpha}e^{-\beta s^{\gamma}}$, we now get
\begin{equation}\label{eq:der}
\left\{\begin{array}{l}
\phi'(s)=\Big(-\dfrac{\alpha}{s}-\gamma\beta s^{\gamma-1}\Big)\phi(s),\\[10pt]
\phi''(s)=\bigg(\dfrac{\alpha}{s^2}-\beta\gamma(\gamma-1)s^{\gamma-2}+\Big(\dfrac{\alpha}{s}+\beta\gamma
s^{\gamma-1}\Big)^{\!2}\bigg)\phi(s).
\end{array}\right.
\end{equation}
By \eqref{eq:phi} and \eqref{eq:der}, with the choices
$\alpha=\alpha_{k}$ with $\alpha_k$ as in \eqref{eq:alfabeta},
$\beta=\frac{1}{4}$, $\gamma=2$ and $\mu=\frac{N-\alpha_{k}}{2}$, it
follows that
$$
v_t-v''-\Big(\dfrac{N-1}{r}\Big)v'+\dfrac{a}{r^2}v-\frac1{r^2}\Delta_{{\mathbb
    S}^{N-1}}v
=0.
$$
In conclusion,
$$
v(x,t)=t^{-\frac{N}{2}+\alpha_k}|x|^{-\alpha_{k}}e^{{-\frac{|x|^2}{4t}}}\psi_k\Big(\frac{x}{|x|}\Big)
$$
is a solution to \eqref{eq:homogeneous}, with initial datum
$v(x,t_0)\in L^1(\R^N)$ (with $t_0>0$). Again we notice that condition
\eqref{eq:nega}, which now reads $a<0$, implies that, taking $k=1$, $v(\cdot,t)\notin
L^\infty(\R^N)$, for any $t>0$, thus  confirming the lack of the classical $L^1-L^\infty$ bound.

Furthermore,
 an improvement of the decay for higher positive modes holds also in
 this case; indeed,
 the above example shows that, for any $M>0$ there exists an initial
 datum $u_0$ (it is enough to take above $k$ large enough) for which the solution
 decays faster that $|t|^{-M}$.

\end{document}